\documentclass[review]{elsarticle}
\usepackage{amsmath}
\usepackage{graphicx,setspace}
\usepackage{amsfonts,amsmath, amssymb, amsthm}
\usepackage{mathrsfs}
\usepackage{epstopdf}
\usepackage{float}
\usepackage{lineno,hyperref}
\usepackage{color}
\usepackage{subfigure}
\usepackage{bm}
\usepackage{graphicx}
\usepackage{amssymb, amsthm}
\usepackage{mathrsfs}
\usepackage{epstopdf}
\usepackage{float}
\usepackage{lineno,hyperref}
\usepackage{color}
\usepackage{subfigure}
\usepackage{cases}

\modulolinenumbers[5]
\numberwithin{equation}{section}

\journal{??}

\begin{document}

\newtheorem{definition}{Definition}
\newtheorem{lemma}{Lemma}
\newtheorem{remark}{Remark}
\newtheorem{theorem}{Theorem}
\newtheorem{proposition}{Proposition}
\newtheorem{assumption}{Assumption}
\newtheorem{example}{Example}
\newtheorem{corollary}{Corollary}
\def\e{\varepsilon}
\def\Rn{\mathbb{R}^{n}}
\def\Rm{\mathbb{R}^{m}}
\def\E{\mathbb{E}}
\def\hte{\bar\theta}
\def\cC{{\mathcal C}}
\numberwithin{equation}{section}

\begin{frontmatter}

\title{{\bf  The persistence of synchronization under $\alpha$-stable noise
}\footnote{{\it AMS Subject Classification(2010):} .}}
\author{\centerline{\bf Yanjie Zhang$^{a,}\footnote{ zhangyj18@scut.edu.cn }$,
Li Lin ${b,*}\footnote{corresponding author: linli@hust.edu.cn }$,
Jinqiao Duan$^{d,}\footnote{duan@iit.edu}$ and
Hongbo Fu$^{e,}\footnote{hbfu@wtu.edu.cn}$}
\centerline{${}^a$ School of Mathematics }
\centerline{South China University of Technology, Guangzhou 510641,  China}
\centerline{${}^b$ Center for Mathematical Sciences }
\centerline{Huazhong University of Science and Technology, Wuhan 430074,  China}
\centerline{${}^d$ Department of Applied Mathematics,} \centerline{Illinois Institute of Technology, Chicago, IL 60616, USA}
\centerline{${}^e$ College of Mathematics and Computer Science,} \centerline{Wuhan Textile University, Wuhan, 430073, PR China}}

\begin{abstract}
This work is about the synchronization of nonlinear coupled dynamical systems driven by $\alpha$-stable noise. Firstly, we provide a novel technique to construct the relationship between synchronized system and slow-fast system.  Secondly, we show that the slow component of original systems converges to the mild solution of the averaging equation under $L^{p}(1<p<\alpha)$ sense. Finally, using the results of averaging principle for stochastic dynamical system with two-time scales, we show that the synchronization effect is persisted  provided equilibria are replaced by stationary random solutions.

\end{abstract}

\begin{keyword}
 Synchronization, $\alpha$-stable noise, averaging principle, random attractor.
\end{keyword}

\end{frontmatter}

%\linenumbers

\section{Introduction}
Synchronization of coupled dissipative systems is a well known phenomenon in biology \cite{sv}, physics \cite{TP1} and social science \cite{S.}. It illustrates that the coupled dynamical systems have a common dynamical behavior in an asymptotic sense. Rodrigues and his coauthors \cite{2,3} investigated mathematically the autonomous systems, including asymptotically stable equilibria and general attractors. They not only showed that the coupled trajectories converged to each other as time increases but also obtained the global attractor of the coupled system. For nonautonomous dynamical system, Afraimovich \cite{1} proved the coupled trajectories converged to each other  with increasing time. Kloeden \cite{4} proved that the coupled trajectories converged to each other as time increased for sufficiently large coupling coefficient and also that the component sets of the pullback attractor of the coupled system converged upper semi-continuously under a uniform global dissipativity condition.

Howerver, nonlinear dynamical systems are subjected to the effect of random fluctuations. The influence of Gaussian noise for synchronization of dissipative dynamical systems has been studied. The random attractors and stochastic stationary solutions were proposed instead of their deterministic counterparts. Caraballo et.al \cite{pe, tc08, tc07} showed that the synchronization of dissipative system persisted when they were disturbed by additive or multiplicative Gaussian noise. Limiting properties of the global random attractor were established as the thinness parameter of the domain $\varepsilon\rightarrow 0$. Flandoli  et.al \cite{ff1, ff2} provided sufficient conditions for synchronization by noise. They proved the existence of a weak point attractor consisting of a single random point for random dynamical systems or order-preserving random dynamical systems. Li et.al \cite{li} presented the convergence rate of synchronization for stochastic differential equations with  nonlinearity multiplicative noise in the mean square sense. Recently, the influence of non-Gaussian noise for synchronization of dissipative dynamical systems has been studied. Liu  et.al \cite{liu} studied the synchronization of dissipative dynamical systems driven by $\alpha$-stable noises, i.e.,
\begin{equation}
\begin{aligned}
dX_t&=\left(f(X_t)+\lambda (\Upsilon_t-X_t)\right)dt+a dL^{1}_t, \\
d\Upsilon_t&=\left(g(X_t)+\lambda (X_t-\Upsilon_t)\right)dt+b dL^{2}_t, \\
\end{aligned}
\end{equation}
where  $a$, $b$  are constant vectors with no components equal to zero, $L^{1}_t$ ,$L^{2}_t$  are independent two-sided scalar L\'evy motion.
Because the integral $\int^{t}_{-\infty}e^{-\lambda(t-s)}dL^{\alpha}_t$ wasn't pathwisely uniformly bounded for $\lambda >1$ on finite time interval $[T_1, T_2]$ (see \cite{ep}),  the random compact absorbing balls weren't contained in the common compact ball. The random attractor was not composed of a singleton set formed by a stationary orbit, so the synchronization phenomenon didn't persist under pathwise sense. Hence, the motivation of this paper is to propose a new and effective method and proves the synchronization phenomenon persisting. Comparing with the known results, the main difficulties here are how to clearly and naturally explain the `averaged' stochastic differential equation and illustrate the persistence of synchronization under $\alpha$-stable noise. In order to overcome these difficulties, we provide a novelty method to study the synchronization problems. By constructing the equivalent relationship between synchronized system and slow-fast system, we transform the problem of synchronization persistence of coupled dynamical systems into discussing the relationship between the stationary solution of two-time scales and stationary solution of `averaged' stochastic dynamical system.

 The theory of averaging principle has a long history in multiscale problem, which was first studied by Khasminskii \cite{kr}, some authors did some generalizations \cite{gj,va,wl}. However, most of the know results in the literature mainly studied the case of Gaussian noise.  Recently, Bao et.al \cite{bg} established the averaging of slow-fast dynamical system driven by $\alpha$-stable noises, where the invariant measure is independent of scale parameter $\varepsilon$ and the  drift coefficient in slow component is uniformly bounded. However, in \cite{bg} it can not cover the case of invariant menasure dependent of scale parameter $\varepsilon$ and more general condition on the drift coefficient in slow component. Hence the another motivation here is how to gain the `averaged' stochastic dynamical system under the case of invariant measure dependent on $\varepsilon$ and discuss the relationship between the stationary solution of two-time scales and stationary solution of `averaged' stochastic dynamical system.

The rest of this paper is organized as follows. In Section 2,  we recall basic concepts about symmetric $\alpha$-stable L\'evy process and random dynamical systems. In Section 3,  we formulate the problem of synchronization of dissipative systems. In Section 4, we show that the slow component of original systems converges to the mild solution of the averaging equation under $L^{p}$ sense. In Section 5, we show that the synchronization effect is persisted  provided equilibria are replaced by stationary random solutions. The paper is concluded in Section 6.

Throughout this paper, generic constants will be denoted by C, whose values may change from one place to another. The constant which depends on  parameter $\gamma$ will be denoted by $C_{\gamma}$.

\section{Preliminaries}
In this section, we recall some basic definitions for symmetric $\alpha $-stable  process  \cite{da, duan} and random dynamical systems \cite{la}.
\subsection{\textbf{Symmetric $\alpha $-stable  process } }
%\begin{definition}
%A stochastic process $L_t$ is a  L\'evy  process if
%\begin{enumerate}
%\item[(1)]$L_0=0$  (a.s.);
%\item[(2)]$L_t$ has independent increments and stationary increments; and
%\item[(3)] $L_t$ has stochastically continuous sample paths, i.e.,  for  every  $s\geq 0$,
% $L(t)\rightarrow L(s)$ in probability, as $t \rightarrow s $.
%\end{enumerate}
%\end{definition}
A L\'evy process $L_t$ taking values in $\mathbb{R}^n$ is characterized by a drift vector $b \in {\mathbb{R}^n}$, an $n \times n$ non-negative-definite, symmetric covariance matrix $ Q $ and a Borel measure $\nu$ defined on ${\mathbb{R}^n}\backslash \{ 0\} $.   We call $(b,Q,\nu)$ the generating triplet of  the L\'evy motions $L_t$  . Moreover, we have the L\'evy-It\^o decomposition for $L_t$ as follows
\begin{equation}
{L_t}= bt + B_{Q}(t) + \int_{|y|< 1} y \widetilde N(t,dy) + \int_{|y|\ge 1} y N(t,dy),
\end{equation}
where $N(dt,dy)$ is the Poisson random measure, $\widetilde N(dt,dy) = N(dt,dy) - \nu (dx)dt$ is the compensated Poisson random measure, $\nu (A) = \mathbb{E}N(1,A)$ is the jump measure, and $ B_{Q}(t)$ is an independent standard $n$-dimensional Brownian motion. The characteristic function of $L_t$ is given by
\begin{equation}
\mathbb{E}[\exp({\rm i}\langle u, L_t \rangle)]=\exp(t\rho(u)), ~~~u \in {\mathbb{R}^n},
\end{equation}
where the function $\rho:{\mathbb{R}^n}\rightarrow \mathbb{C}$ is  the characteristic exponent
\begin{equation}
\rho(u)={\rm i}\langle u, b\rangle-\frac{1}{2}\langle u, Qu \rangle+\int_{{\mathbb{R}^n}\backslash \{ 0\}}{(e^{{\rm i}\langle u, z \rangle}-1-{\rm i}\langle u,z\rangle {I_{\{ | z | \textless 1\} }})\nu(dz)}.
\end{equation}
The Borel measure $\nu $ is called the jump measure. Here  $|\cdot|$ be the Euclidean norm, $\langle \cdot, \cdot \rangle $ be the scalar product in $ \mathbb{R}^n $.
\begin{definition}
For $\alpha \in (0,2)$, an $n$-dimensional symmetric $\alpha $-stable process $ L^{\alpha}_{t} $ is a L\'evy process with characteristic exponent $\rho$
\begin{equation}
\rho(u)=-C_1(n,\alpha)| u |^{\alpha},  ~for~u \in {\mathbb{R}^{n}}
\end{equation}
with $ C_1(n, \alpha):=\pi^{-\frac{1}{2}}\Gamma((1+\alpha)/2)\Gamma(n/2)/\Gamma((n+\alpha)/2)$.
\end{definition}

For an $n$-dimensional symmetric $\alpha$-stable L\'evy process, the diffusion matrix $ Q= 0$,
the drift vector $ b= 0$, and the L\'evy measure $\nu $ is given by
\begin{equation}
\nu(du)=\frac{C(n,\alpha)}{{| u |}^{n+\alpha}}du,
\end{equation}
where $ C(n, \alpha):=\alpha\Gamma((n+\alpha)/2)/{(2^{1-\alpha}\pi^{n/2}\Gamma(1-\alpha/2))}$.
\subsection{\textbf{Random dynamical systems}}
%\begin{definition}
%Let $(\Omega,\mathscr{F},\mathbb{P})$ be a probability space, and $(\theta_t)_{t\in\mathbb{R}}$ a family of measurable
%transformations from $\Omega$ to $\Omega$.
%We call $(\Omega,\mathscr{F},\mathbb{P}; (\theta_t)_{t\in\mathbb{R}})$ a metric dynamical system if for each $t\in\mathbb{R}$,
%$\theta_t$ preserves the probability measure $\mathbb{P}$, i.e.,
%\begin{equation}
%\theta_t^*\mathbb{P}=\mathbb{P},
%\end{equation}
%and for $s,t\in\mathbb{R}$,
%\begin{equation}
%\theta_0=1_\Omega, \quad \theta_{t+s}=\theta_t\circ\theta_s.
%\end{equation}
%\end{definition}
\begin{definition}
Let $ (H,\mathscr{B}(H))$ be a measurable space. A random dynamical system over a  metric dynamical system $(\Omega,\mathscr{F},\mathbb{P},(\theta_{t})_{t\in\mathbb{R}})$ with time space $\mathbb{R^{+}}$ is given by a mapping
\begin{equation*}
    \phi: \mathbb{R^{+}}\times\Omega\times H \to H,
\end{equation*}
that is jointly $\mathcal{B}(\mathbb{R^{+}})\otimes\mathscr{F}\otimes\mathcal{B}(H)/\mathcal{B}(H)-$ measurable and satisfies
the cocycle property:
\begin{equation}
\label{cocycle}
\begin{aligned}
\phi(0,\omega,\cdot)&={\rm id}_{H}, \text{ for each}\ \omega\in\Omega, \\
\phi(t+s,\omega,\cdot)&=\phi\big(t,\theta_s\omega,\phi(s,\omega,\cdot)\big),\ \text{ for each}\ s,\,t\in\mathbb{R^{+}},\ \omega\in\Omega.
\end{aligned}
\end{equation}
\begin{definition}
A random variable $\omega\mapsto X(\omega)$ with values in $H$ is called a stationary orbit (or random
fixed point) for a random dynamical system $\phi$ if
\begin{equation}
\phi(t,\omega,X(\omega))=X(\theta_{t}\omega),\ \text{ for }\ t\in\mathbb{R^{+}},\ \omega\in\Omega.
\end{equation}
\end{definition}
\end{definition}
\begin{definition}
A random variable $X$ is called tempered if
\begin{equation}
\lim_{t\rightarrow +\infty} \frac{\log^{+}|X(\theta_{-t}\omega)|}{t}=0, ~~a.s. ~~\omega.
\end{equation}
\end{definition}
%In other words, for $\omega\in {\Omega \backslash N } $ and $\varepsilon > 0$, where $N \subset \Omega$  is a set with $\mathbb{P}(N)=0$, there exists a $t_0 (\varepsilon, \omega)\geq 0$ such that for $t \geq t_0(\varepsilon, \omega)$, it holds that
%\begin{equation}
%|X(\theta_{-t}\omega)|\leq e^{\varepsilon t}.
%\end{equation}
A random set $\mathcal{A}(\omega)\subseteq \mathbb{R}^{n}$ is called tempered if the random variable $\sup_{x \in \mathcal{A}(\omega)}|x|$ is tempered.

\begin{definition}
 A compact random set $\mathcal{A}\in \mathcal{D}$ is called a weak attractor of $\varphi$ if for all $\omega \in \Omega$, $t>0$ and closed tempered random set $D(\omega)\subseteq \mathbb{R}^{n}$, we have
\begin{enumerate}
\item[(i)]$\varphi\left(t, \omega ,\mathcal{A}(\omega)\right)=\mathcal{A}(\theta_{t}\omega)$;
\item[(ii)]$(l.i.p)~\lim_{t\rightarrow \infty}dist_{H}\left(\varphi\left(t,\theta_{-t}\omega, D(\theta_{-t}\omega)\right), \mathcal{A}(\omega)\right)=0$,
\end{enumerate}
where $l.i.p$ denotes limit in probability.
\end{definition}

\section{Formulation of the problem}

We formally define the synchronization for a given random dynamical system.
\begin{definition}
We say that the synchronization occurs if $\mathcal{A}(\omega)$ is a singleton, for $\mathbb{P}$-a.e. $\omega \in \Omega$.
\end{definition}
Consider the following  stochastic dynamical systems
\begin{equation}
\label{sde}
\left\{
\begin{aligned}
dX_t&=f(X_t)dt+\sigma_1dL^{\alpha}_t, \\
dY_t&=g(Y_t)dt+\sigma_2dL^{\alpha}_t, \\
\end{aligned}
\right.
\end{equation}
where $(X_t, Y_t)$ is an $\mathbb{R}^n \times \mathbb{R}^n$-valued process, $f,g$ are drift terms (vector fields), $\sigma_1$ and $\sigma_2$ are non-zero real noise intensities, and $L^{\alpha}_t$  (with $1<\alpha<2$) is symmetric $\alpha$-stable L\'evy process with triplets $(0,0,{\nu_{\alpha}})$.

The synchronized system corresponding to stochastic dynamical systems \eqref{sde} reads
\begin{equation}
\label{syn}
\left\{
\begin{aligned}
dX_t&=\big(f(X_t)+\nu(Y_t-X_t)\big)dt+\sigma_1dL^{\alpha}_t, \\
dY_t&=\big(g(Y_t)+\nu(X_t-Y_t)\big)dt+\sigma_2dL^{\alpha}_t, \\
\end{aligned}
\right.
\end{equation}
where $\nu >0$.

The aim is to show that this synchronization effect is preserved under additive $\alpha$-stable noise provided equilibria are replaced by stationary random solutions, and obtain the convergence rate of synchronization of the coupled systems. That is to say we will prove that the synchronization systems \eqref{syn} have a unique stochastic stationary solution $(\bar{X}^{\nu}_t, \bar{Y}^{\nu}_t)$, which is globally asymptotically stable with
\begin{equation}
(\bar{X}^{\nu}_t, \bar{Y}^{\nu}_t)\rightarrow (\hat{X}_t, \hat{X}_t), ~~as ~~\nu \rightarrow \infty,
\end{equation}
on finite time interval $[T_1, T_2]$ of $\mathbb{R}$, where $\hat{X}_t$ is the unique globally asymptotically stable stationary solution of the `averaged' SDE
\begin{equation}
d\bar{X}_t=\frac{1}{2}\left[f(\bar{X}_t)+g(\bar{X}_t)\right]dt +\frac{\sigma_1+\sigma_2}{2}dL^{\alpha}_{t}.
\end{equation}
In the following, we will propose a new and effective method aiming at proving the synchronization phenomenon persisting.

\section{Stochastic averaging}
Introduce the following transformation
\begin{equation}
\label{transmation}
\left\{
\begin{aligned}
X_t&=X^{\varepsilon}_t+\varepsilon^{\frac{1}{\alpha}}Y^{\varepsilon}_t, \\
Y_t&=X^{\varepsilon}_t-\varepsilon^{\frac{1}{\alpha}}Y^{\varepsilon}_t, \\
\nu&=\frac{1}{\varepsilon},
\end{aligned}
\right.
\end{equation}
then the equations \eqref{syn} can be rewritten as
\begin{subequations}
\begin{equation}
\label{slo}
dX^{\varepsilon}_t=\frac{1}{2}\left[f(X^{\varepsilon}_t+\varepsilon^{\frac{1}{\alpha}}Y^{\varepsilon}_t)
+g(X^{\varepsilon}_t-\varepsilon^{\frac{1}{\alpha}}Y^{\varepsilon}_t)\right]dt+\frac{(\sigma_1+\sigma_2)}{2}dL^{\alpha}_t,
\end{equation}
\begin{equation}
\label{sloy}
dY^{\varepsilon}_t=\frac{1}{\varepsilon}\left[\frac{1}{2}\varepsilon^{1-\frac{1}{\alpha}}\big(f(X^{\varepsilon}_t
+\varepsilon^{\frac{1}{\alpha}}Y^{\varepsilon}_t)-g(X^{\varepsilon}_t-\varepsilon^{\frac{1}{\alpha}}Y^{\varepsilon}_t)\big)\right]dt-\frac{2}{\varepsilon}Y^{\varepsilon}_tdt+\frac{(\sigma_1-\sigma_2)}{2\varepsilon^{\frac{1}{\alpha}}}dL^{\alpha}_t.
\end{equation}
\end{subequations}
Obviously, the equations \eqref{slo}-\eqref{sloy} can be viewed as the slow-fast stochastic dynamical system.

Now we impose the following assumptions on the coefficients $f, g$ for the slow-fast stochastic dynamical system \eqref{slo}-\eqref{sloy}.

{\bf Hypothesis  H.1 }(i)
The functions $f, g $, viewed as functions of $(x, y, \varepsilon)$, satisfy the global Lipschitz conditions,
i.e., there exists a positive constant $L$ such that
\begin{equation}
\begin{aligned}
\left|f(x_1, y_1,\varepsilon)-f(x_2, y_2,\varepsilon)\right|& \leq L (|x_1-x_2|+|y_1-y_2|),\\
\left|g(x_1, y_1,\varepsilon)-g(x_2, y_2,\varepsilon)\right|& \leq L (|x_1-x_2|+|y_1-y_2|).\\
\end{aligned}
\end{equation}
(ii) The functions $f, g $ satisfy the linear growth conditions,
i.e., there exists a positive constant $M_1$ such that
\begin{equation}
\begin{aligned}
\left|f(x, y,\varepsilon)\right|& \leq M_1 (1+|x|+|y|),\\
\left|g(x, y,\varepsilon)\right|& \leq M_1 (1+|x|+|y|).\\
\end{aligned}
\end{equation}
\begin{remark}
Under Hypothesis $\bf{H.1}$,  there exists a unique solution $\{(X^{\varepsilon}_t, Y^{\varepsilon}_t ), t\geq 0\}$ to system \eqref{slo}-\eqref{sloy}.
\end{remark}

\par
{\bf Hypothesis  H.2 }
There exist positive constants $M_2$ and $R$, such that for any $(x,\varepsilon)$ and $|y|\geq {R}$,
\begin{equation}
\label{expo}
\left\langle y,f(x,y,\varepsilon)-g(x,y,\varepsilon)\right\rangle \leq -M_2 | y |^{2}.
\end{equation}
\begin{remark}
The condition \eqref{expo} ensures the existence of an invariant measure  $\mu^{\varepsilon}_{x}(dy)$  for the fast component $Y^{\varepsilon}_t$ with $X^{\varepsilon}_t=x$. Moreover, this invariant measure is dependent of $\varepsilon$ (see \cite[Theorem 1.1]{wj}).
\end{remark}

\par
{\bf Hypothesis  H.3 }
There exist positive constants $M_3, M_4, M_5$ and $M_6$,  such that
\begin{equation}
\label{0b}
\begin{aligned}
&\sup_{x, y,\varepsilon}|f(x,y,\varepsilon)|\leq M_3, ~~
\sup_{x, \varepsilon}|g(x,y,\varepsilon)|\leq M_4(1+|y|), \\
&\sup_{y,\varepsilon}|\nabla_{x}f(x,y,\varepsilon)|\leq M_5,
~~\sup_{y, \varepsilon}|\nabla_{x}g(x,y,\varepsilon)|\leq M_6.
\end{aligned}
\end{equation}

Denote $F(x,y,\varepsilon):=f(x+\varepsilon^{\frac{1}{\alpha}}y)$ and $G(x,y,\varepsilon):=g(x-\varepsilon^{\frac{1}{\alpha}}y)$, then the slow-fast stochastic dynamical system \eqref{slo} and \eqref{sloy} can be written as
\begin{equation}
\label{slow1}
\left\{
\begin{aligned}
dX^{\varepsilon}_t&=\frac{1}{2}\left[F\left(X^{\varepsilon}_t, Y^{\varepsilon}_t, \varepsilon)
+G(X^{\varepsilon}_t,Y^{\varepsilon}_t, \varepsilon\right)\right]dt+\frac{\sigma_1+\sigma_2}{2}dL^{\alpha}_t, \\
dY^{\varepsilon}_t&=\frac{1}{\varepsilon}\left[\frac{1}{2}\varepsilon^{1-\frac{1}{\alpha}}\left(F\left(X^{\varepsilon}_t, Y^{\varepsilon}_t, \varepsilon\right)-G\left(X^{\varepsilon}_t,Y^{\varepsilon}_t,\varepsilon\right)\right)\right]dt-\frac{2}{\varepsilon}Y^{\varepsilon}_tdt+\frac{\sigma_1-\sigma_2}{2\varepsilon^{\frac{1}{\alpha}}}dL^{\alpha}_t.
\end{aligned}
\right.
\end{equation}

\subsection{Some priori estimates of $(X^{\varepsilon}_t, Y^{\varepsilon}_t)$}
In this subsection, we prove some uniform bounds for the moments of the solution $(X^{\varepsilon}_t, Y^{\varepsilon}_t )$.
\begin{lemma}
\label{bounded}
Under Hypotheses $\bf{H.1}$-$\bf{H.3}$, for any constant $T>0$, there exist constants $C_{T}> 0$ and $\varepsilon_0> 0$, such that  for $1<p< \alpha$ and $\varepsilon \in (0, \varepsilon_0 )$,  we have
\begin{equation}
\sup_{0\leq t\leq T}(\mathbb{E}(X^{\varepsilon}_t)^{p})^{\frac{1}{p}}\leq C_{T}(1+|x_0|),~~ \sup_{0\leq t \leq T}\left(\mathbb{E}|Y^{\varepsilon}_t|^{p}\right)^{\frac{1}{p}} \leq C_{T}(1+|y_0|).
\end{equation}
\begin{proof}
It's easy to know that $\forall$ $ 1< p < \alpha$ and $0\leq t\leq T$, we have
\begin{equation}
\label{moment}
\sup_{0\leq t \leq T}\mathbb{E}|L^{\alpha}_t|^{p}<\infty.
\end{equation}
Set
\begin{equation}
Z^{\varepsilon}_{t}=\frac{1}{\varepsilon^{\frac{1}{\alpha}}}\int^{t}_{0}e^{\frac{-2(t-s)}{\varepsilon}}dL^{\alpha}_{s}.
\end{equation}
By the similar method \cite[Theorem 4.4]{ep}, we know%onsider a new probability space $(\Omega^{'}, \mathcal{F}^{'}, \mathbb{P}^{'})$ and a Rademacher sequence $r_{i}$, which are defined by
\begin{equation}
\begin{aligned}
\label{intro}
\mathbb{E}|Z^{\varepsilon}_{t}|^{p}&\leq C_{\alpha,p}.
\end{aligned}
\end{equation}
Obviously,
\begin{equation}
\begin{aligned}
Y^{\varepsilon}_t&=e^{-\frac{2}{\varepsilon}t}y_0+\frac{1}{\varepsilon}\int^{t}_{0}e^{-\frac{2}{\varepsilon}(t-s)}\left[\frac{1}{2}\varepsilon^{1-\frac{1}{\alpha}}
\left(F(X^{\varepsilon}_s,Y^{\varepsilon}_s,\varepsilon)-
G(X^{\varepsilon}_s,Y^{\varepsilon}_s,\varepsilon)\right)\right]ds \\
&+\frac{\sigma_1-\sigma_2}{2\varepsilon^{\frac{1}{\alpha}}}\int^{t}_{0}e^{-\frac{2}{\varepsilon}(t-s)}dL^{\alpha}_{s}.
\end{aligned}
\end{equation}
By Minkowski's inequality and Hypothesis $\bf{H.3}$, we have
\begin{equation}
\begin{aligned}
\left(\mathbb{E}|Y^{\varepsilon}_t|^{p}\right)^{\frac{1}{p}}&\leq |e^{-\frac{2}{\varepsilon}t}y_0|+\frac{1}{\varepsilon}\int^{t}_{0}e^{\frac{2(s-t)}{\varepsilon}}\left[\mathbb{E}\left|\frac{1}{2}\varepsilon^{1-\frac{1}{\alpha}}
\left(F\left(X^{\varepsilon}_s,Y^{\varepsilon}_s,\varepsilon)-
G(X^{\varepsilon}_s,Y^{\varepsilon}_s,\varepsilon\right)\right)\right|^{p}\right]^{\frac{1}{p}}ds\\
&+\frac{\sigma_1-\sigma_2}{2}\left(\mathbb{E}|Z^{\varepsilon}_{t}|^{p}\right)^{\frac{1}{p}}\\
&\leq |y_0|+\frac{1}{\varepsilon}\int^{t}_{0}e^{\frac{2(s-t)}{\varepsilon}}\left(\frac{1}{2}M_3\varepsilon^{1-\frac{1}{\alpha}}
+\frac{1}{2}M_4\varepsilon^{1-\frac{1}{\alpha}}+\frac{1}{2}M_4\varepsilon^{1-\frac{1}{\alpha}}\left(\mathbb{E}|Y^{\varepsilon}_s|^{p}\right)^{\frac{1}{p}}\right)ds\\
&+(\sigma_1-\sigma_2)\left(\mathbb{E}|Z^{\varepsilon}_{t}|^{p}\right)^{\frac{1}{p}}\\
&\leq |y_0|+\frac{1}{\varepsilon}\int^{t}_{0}e^{\frac{2(s-t)}{\varepsilon}}\left(C+\frac{1}{2}M_4\varepsilon^{1-\frac{1}{\alpha}}\left(\mathbb{E}|Y^{\varepsilon}_s|^{p}\right)^{\frac{1}{p}}\right)ds\\
&+(\sigma_1-\sigma_2)\left(\mathbb{E}|Z^{\varepsilon}_{t}|^{p}\right)^{\frac{1}{p}},
\end{aligned}
\end{equation}
where $C$ is a positive constant independent of $\varepsilon$. Using Gr\"onwall inequality and \eqref{intro}, we get
\begin{equation}
\sup_{0\leq t \leq T}\left(\mathbb{E}|Y^{\varepsilon}_t|^{p}\right)^{\frac{1}{p}} \leq C_{T}(1+|y_0|).
\end{equation}
Similarly,  we have
\begin{equation}
\sup_{0\leq t\leq T}(\mathbb{E}(X^{\varepsilon}_t)^{p})^{\frac{1}{p}}\leq C_{T}(1+|x_0|).
\end{equation}
\end{proof}
\end{lemma}
Stationary solutions of stochastic dynamical systems describe the invariance over time along a measurable and measure-preserving transformation, and the long time limit for the solutions of these systems. A stationary solution means that the finite-dimensional distributions of the solution are independent of shifts with respect to the time.

Define
\begin{equation}
\label{c0}
\begin{aligned}
Y^{\varepsilon}_{t,x}(y_0)&=e^{\frac{-2t}{\varepsilon}}y_0+\frac{1}{\varepsilon}\int^{t}_{0}e^{-\frac{2}{\varepsilon}(t-s)}
\left[\frac{1}{2}\varepsilon^{1-\frac{1}{\alpha}}\left(F\left(x,Y^{\varepsilon}_{s,x}(y_0),\varepsilon \right)-G\left(x,Y^{\varepsilon}_{s,x}(y_0),\varepsilon\right)\right)\right]ds \\ &+\frac{\sigma_1-\sigma_2}{2\varepsilon^{\frac{1}{\alpha}}}\int^{t}_{0}e^{-\frac{2}{\varepsilon}(t-s)}dL^{\alpha}_{s}.
\end{aligned}
\end{equation}
\begin{lemma}
Under Hypotheses $\bf{H.1}$-$\bf{H.3}$, then for any fixed $(x,\varepsilon)$, The Eq.\eqref{c0} has a unique stationary solution.
\end{lemma}
\begin{proof}
The result comes from \cite[Theorem 3.6]{qh}.
\end{proof}
\begin{lemma}
\label{avv}
Under Hypotheses $\bf{H.1}$-$\bf{H.3}$, for any constant $T>0$, there exist positive constants $C_{T}$ and  $\varepsilon_0$, such that for each fixed $x$, $\varepsilon \in(0, \varepsilon_0)$ and $F\in C^{1}_{b}(\mathbb{R}^{n})$, we have
\begin{equation}
\begin{aligned}
\left|\mathbb{E}\left(F(Y^{\varepsilon}_{t,x})\right)-\int_{\mathbb{R}^{n}}F(z)\mu^{\varepsilon}_{x}(dz)\right|\leq C_{T} \left[ e^{-\frac{2}{\varepsilon}t}\left(|y|+|z|\right)\right].
\end{aligned}
\end{equation}
\begin{proof}
By  Lemma 1,  we know
\begin{equation}
\label{c1}
\sup_{0\leq t\leq T}\left(\mathbb{E}|Y^{\varepsilon}_{t,x}(y_0)|^{p}\right)^{\frac{1}{p}}\leq C_{T}( 1+|y_0|), ~~p\in (1, \alpha).
\end{equation}
By Hypothesis $\bf{H.1}$, we have
\begin{equation}
\label{c2}
\begin{aligned}
&\left(\mathbb{E}|Y^{\varepsilon}_{t,x}(y_1)-Y^{\varepsilon}_{t,x}(y_2)|^{p}\right)^{\frac{1}{p}}\\
&\leq e^{-\frac{2}{\varepsilon}t}|y_1-y_2|+\frac{1}{2}\varepsilon^{1-\frac{1}{\alpha}}\left[\frac{1}{\varepsilon}\int^{t}_{0}e^{-\frac{2}{\varepsilon}(t-s)}
\left(\mathbb{E}|F\left(x, Y^{\varepsilon}_{s,x}(y_1), \varepsilon\right)-F\left(x, Y^{\varepsilon}_{s,x}(y_2), \varepsilon\right)|^{p}\right)^{\frac{1}{p}}ds \right]\\
&+\frac{1}{2}\varepsilon^{1-\frac{1}{\alpha}}\left[\frac{1}{\varepsilon}\int^{t}_{0}e^{-\frac{2}{\varepsilon}(t-s)}
\left(\mathbb{E}|G(x, Y^{\varepsilon}_{s,x}(y_2), \varepsilon)-G(x, Y^{\varepsilon}_{s,x}(y_1), \varepsilon)|^{p}\right)^{\frac{1}{p}}ds\right] \\
&\leq e^{-\frac{2}{\varepsilon}t}|y_1-y_2|+L \varepsilon^{1-\frac{1}{\alpha}} \left[\frac{1}{\varepsilon} \int^{t}_{0}e^{-\frac{2}{\varepsilon}(t-s)}\left(\mathbb{E}|Y^{\varepsilon}_{s,x}(y_2)-Y^{\varepsilon}_{s,x}(y_1)|^{p}\right)^{\frac{1}{p}}ds\right].
\end{aligned}
\end{equation}
By Gr\"onwall inequality, we have

\begin{equation}
\label{c3}
\left(\mathbb{E}|Y^{\varepsilon}_{t,x}(y_1)-Y^{\varepsilon}_{t,x}(y_2)|^{p}\right)^{\frac{1}{p}}\leq C_{T} e^{-\frac{2t}{\varepsilon}}|y_1-y_2|.
\end{equation}

Combined with \eqref{c1} and \eqref{c2}, we have
\begin{equation}
\begin{aligned}
\left(\mathbb{E}|Y^{\varepsilon}_{t,x}(y)|^{p}\right)^{\frac{1}{p}}&\leq \left(\mathbb{E}|Y^{\varepsilon}_{t,x}(0)|^{p}\right)^{\frac{1}{p}}+\left(\mathbb{E}|Y^{\varepsilon}_{t,x}(y)-Y^{\varepsilon}_{t,x}(0)|^{p}\right)^{\frac{1}{p}}\\
&\leq C_{T}\left(1+e^{-\frac{2t}{\varepsilon}}|y|\right).
\end{aligned}
\end{equation}
The stationary solution of equation \eqref{c0} can be denoted by $\check{Y}^{\varepsilon}_{t,x}(y)$, then we have
\begin{equation}
\begin{aligned}
\left(\int_{\mathbb{R}^{n}}|z|^{p}\mu^{\varepsilon}_{x}(dz)\right)^{\frac{1}{p}}&= \left(\mathbb{E}|\check{Y}^{\varepsilon}_{t,x}(y)|^{p}\right)^{\frac{1}{p}}\leq C_{T}\left(1+e^{\frac{-2t}{\varepsilon}}|y|\right).
\end{aligned}
\end{equation}
Let $t\rightarrow \infty$, we have
\begin{equation}
\left(\int_{\mathbb{R}^{n}}|y|^{p}\mu^{\varepsilon}_{x}(dy)\right)^{\frac{1}{p}}\leq C_{T}.
\end{equation}
Therefore, by Lemma 5, Hypothesis $\bf{H.1}$ and \eqref{c3},  we have
\begin{equation}
\begin{aligned}
\left|\mathbb{E}\left[F(x,Y^{\varepsilon}_{t,x}(y),\varepsilon)\right]-{F}(x,\varepsilon)\right|
&=\left|\mathbb{E}\left[F(x,Y^{\varepsilon}_{t,x}(y),\varepsilon)\right]-
\int_{\mathbb{R}^{n}} F(x,y,\varepsilon)\mu^{\varepsilon}_{x}(dy)\right|\\
&=\left|\mathbb{E}\left[F(x,Y^{\varepsilon}_{t,x}(y),\varepsilon)\right]-\mathbb{E}\left[F(x,\check{Y}^{\varepsilon}_{t,x}(z),\varepsilon)\right]\right|\\
& \leq  L \left(\mathbb{E}|Y^{\varepsilon}_{t,x}(y)-\check{Y}^{\varepsilon}_{t,x}(z)|^{p}\right)^{\frac{1}{p}}\\
& \leq L C_{T} \left[ e^{-\frac{2}{\varepsilon}t}|y-z|\right]\\
&\leq C_{T} \left[ e^{-\frac{2}{\varepsilon}t}\left(|y|+|z|\right)\right].\\
\end{aligned}
\end{equation}
\end{proof}
\end{lemma}

\begin{lemma}
\label{path}
Under Hypotheses $\bf{H.1}$-$\bf{H.3}$, then for any constant $T>0$, $h\in(0,1)$ and $p\in(1, \alpha)$, there exists a positive constant $C_{T}$, such that
\begin{equation}
\sup_{0\leq t \leq T}\left(\mathbb{E}\left|X^{\varepsilon}_{t+h}-X^{\varepsilon}_{t}|^{p}\right|\right)^{\frac{1}{p}}\leq C_{T} h^{\theta},
\end{equation}
where $\theta \in (0,1)$.
\end{lemma}
\begin{proof}
We have
\begin{equation}
X^{\varepsilon}_{t+h}-X^{\varepsilon}_t=\int^{t+h}_t\left[F(X^{\varepsilon}_s, Y^{\varepsilon}_s, \varepsilon )+G(X^{\varepsilon}_s, Y^{\varepsilon}_s, \varepsilon )\right]ds+\frac{\sigma_1+\sigma_2}{2}(L^{\alpha}_{t+h}-L^{\alpha}_{t}).
\end{equation}
By condition $\eqref{0b}$, Lemma 1 and structural properties of stable process,  we have
\begin{equation}
\sup_{0\leq t \leq T}\left(\mathbb{E}\left|X^{\varepsilon}_{t+h}-X^{\varepsilon}_{t}|^{p}\right|\right)^{\frac{1}{p}}\leq C_{T} h^{\theta}.
\end{equation}

\end{proof}
\subsection{Estimates of the auxiliary processes}
In order to get the averaging equation, we need to introduce the following auxiliary processes.
\begin{equation}
\label{xy}
\left\{
\begin{aligned}
\tilde{X}^{\varepsilon}_{t}&=x_0+\frac{1}{2}\int_{0}^{t}\left[F\left({X}^{\varepsilon}_{[s/\delta]\delta},\tilde{Y}^{\varepsilon}_{s},\varepsilon\right)
+G\left({X}^{\varepsilon}_{[s/\delta]\delta},\tilde{Y}^{\varepsilon}_{s},\varepsilon\right)\right]ds+\frac{\sigma_1+\sigma_2}{2}L^{\alpha}_t\\
\tilde{Y}^{\varepsilon}_{t}&=e^{\frac{-2t}{\varepsilon}}y_0+\frac{1}{\varepsilon}\int_{0}^{t}e^{\frac{-2(t-s)}{\varepsilon}}
\left[\frac{1}{2}\varepsilon^{1-\frac{1}{\alpha}}\left(F\left({X}^{\varepsilon}_{[{s}/{\delta}]\delta},\tilde{Y}^{\varepsilon}_s,\varepsilon\right)
-G\left({X}^{\varepsilon}_{[{s}/{\delta}]\delta},\tilde{Y}^{\varepsilon}_s,\varepsilon\right)\right)\right]ds \\
&+\frac{\sigma_1-\sigma_2}{2\varepsilon^{\frac{1}{\alpha}}}\int_{0}^{t}e^{\frac{-2(t-s)}{\varepsilon}}dL^{\alpha}_{s}.
\end{aligned}
\right.
\end{equation}
\begin{lemma}
\label{lem5}
Under Hypotheses $\bf{H.1}$-$\bf{H.3}$, then for constant $T>0$ and $p \in (1,\alpha)$, there exist positive constants $\varepsilon_0$ and $C_{T}$, such that
for any $\varepsilon \in (0, \varepsilon_0)$ , we have
\begin{equation}
\int_{0}^{T}\left(\mathbb{E}|Y^{\varepsilon}_{t}-\tilde{Y}^{\varepsilon}_t|^{p}\right)^{\frac{1}{p}}dt
\leq C_{T}\left(\frac{\varepsilon}{\delta}+\varepsilon {\delta}^{-(1-\theta)}e^{C\delta/{\varepsilon}}\right),
\end{equation}
and
\begin{equation}
\int_{0}^{T}\left(\mathbb{E}|X^{\varepsilon}_{t}-\tilde{X}^{\varepsilon}_t|^{p}\right)^{\frac{1}{p}}dt
\leq C_{T}\left(\delta^{\theta}+\frac{\varepsilon}{\delta}+\varepsilon \delta^{-(1-\theta)}e^{C\delta/\varepsilon}\right).
\end{equation}
\begin{proof}
By the similar method of Lemma 2, for any $t\in [0, T]$, we have
\begin{equation}
\label{bounded}
\begin{aligned}
\sup_{0\leq t\leq T}\mathbb{E}|\tilde{X}^{\varepsilon}_{t}|^{p}<\infty, ~~\sup_{0\leq t\leq T}\mathbb{E}|\tilde{Y}^{\varepsilon}_{t}|^{p}<\infty.
\end{aligned}
\end{equation}
For any $t\in [0, T]$, there is a positive integer $k$, such that $t\in\left[k\delta, (k+1)\delta\right)$. Set $\Psi^{\varepsilon}_{t}:={Y}^{\varepsilon}_{t}-\tilde{Y}^{\varepsilon}_{t}$. Then by Hypothesis $\bf{H.1}$, we have
\begin{equation}
\label{xx}
\begin{aligned}
&\left(\mathbb{E}|{X}^{\varepsilon}_{t}-\tilde{X}^{\varepsilon}_{t}|^{p}\right)^{\frac{1}{p}}\\
&\leq \int^{t}_{0}\left(\mathbb{E}\left|F\left(X^{\varepsilon}_{s},Y^{\varepsilon}_{s},\varepsilon\right)+G\left(X^{\varepsilon}_{s},Y^{\varepsilon}_{s},\varepsilon\right)
-F({X}^{\varepsilon}_{[s/\delta]\delta},\tilde{Y}^{\varepsilon}_{s},\varepsilon)-G({X}^{\varepsilon}_{[s/\delta]\delta},\tilde{Y}^{\varepsilon}_{s},\varepsilon)\right|^{p}\right)^{\frac{1}{p}}ds\\
&\leq 2L\int^{t}_{0}\left(\mathbb{E}|{X}^{\varepsilon}_{s}-{X}^{\varepsilon}_{[s/\delta]\delta}|^{p}\right)^{\frac{1}{p}}ds
+2L\int^{t}_{0}\left(\mathbb{E}|\Psi^{\varepsilon}_{s}|^{p}\right)^{\frac{1}{p}}ds.
\end{aligned}
\end{equation}
On the one hand,
\begin{equation}
\begin{aligned}
&\left(\mathbb{E}|{X}^{\varepsilon}_{s}-{X}^{\varepsilon}_{[s/\delta]\delta}|^{p}\right)^{\frac{1}{p}}\\
&\leq \frac{1}{2}\int^{s}_{[s/\delta]\delta}\left(\mathbb{E}\left|F(X^{\varepsilon}_{u}, Y^{\varepsilon}_{u},\varepsilon)+
G(X^{\varepsilon}_{u},Y^{\varepsilon}_{u},\varepsilon)\right|^{p}\right)^{\frac{1}{p}}du+\frac{\sigma_1+\sigma_2}{2}\left(\mathbb{E}\left|L^{\alpha}_{s-[s/\delta]\delta}\right|^{p}\right)^{\frac{1}{p}} \\
&:=\Phi_1(t)+\Phi_2(t).
\end{aligned}
\end{equation}
For $\Phi_1(t)$ and $\Phi_2(t)$,  by Hypothesis $\bf{H.3}$ and \eqref{moment} respectively, we have
\begin{equation}
\label{m1}
\Phi_2(t)\leq C\delta^{\theta}, ~~\Phi_1(t)\leq C\delta^{\theta}.
\end{equation}
Therefore for any $t\in [0, T]$, we have
\begin{equation}
\label{er5}
\int^{t}_{0}\left(\mathbb{E}|{X}^{\varepsilon}_{s}-{X}^{\varepsilon}_{[s/\delta]\delta}|^{p}\right)^{\frac{1}{p}}ds \leq C_{T}\delta^{\theta}.
\end{equation}
On the other hand, by Hypothesis $\bf{H.1}$,   we have
\begin{equation}
\begin{aligned}
\left(\mathbb{E}|\Psi^{\varepsilon}_{t}|^{p}\right)^{\frac{1}{p}}&\leq e^{\frac{-2(t-k\delta)}{\varepsilon}}\left(\mathbb{E}|\Psi^{\varepsilon}_{k\delta}|^{p}\right)^{\frac{1}{p}}+\frac{1}{\varepsilon}\int^{t}_{k\delta}e^{-\frac{2(t-s)}{\varepsilon}}\frac{1}{2}\varepsilon^{1-\frac{1}{\alpha}}
\left(\mathbb{E}\left|F(X^{\varepsilon}_{s},Y^{\varepsilon}_{s},\varepsilon)-F(X^{\varepsilon}_{k\delta},\tilde{Y}^{\varepsilon}_{s},\varepsilon)\right|^{p}\right)^{\frac{1}{p}}ds\\
&+\frac{1}{\varepsilon}\int^{t}_{k\delta}e^{-\frac{2(t-s)}{\varepsilon}}\frac{1}{2}\varepsilon^{1-\frac{1}{\alpha}}
\left(\mathbb{E}\left|G(X^{\varepsilon}_{s},Y^{\varepsilon}_{s},\varepsilon)-G(X^{\varepsilon}_{k\delta},\tilde{Y}^{\varepsilon}_{s},\varepsilon)\right|^{p}\right)^{\frac{1}{p}}ds\\
&\leq e^{\frac{-2(t-k\delta)}{\varepsilon}}\left(\mathbb{E}|\Psi^{\varepsilon}_{k\delta}|^{p}\right)^{\frac{1}{p}}
+\frac{1}{\varepsilon}\int^{t}_{k\delta}e^{-\frac{2(t-s)}{\varepsilon}}\left\{L\varepsilon^{1-\frac{1}{\alpha}}
\left(\mathbb{E}\left|X^{\varepsilon}_{s}-X^{\varepsilon}_{k\delta}\right|^{p}\right)^{\frac{1}{p}}+L\varepsilon^{1-\frac{1}{\alpha}}
\left(\mathbb{E}\left|\Psi^{\varepsilon}_{s}\right|^{p}\right)^{\frac{1}{p}}\right\}ds.
\end{aligned}
\end{equation}
By Gr\"onwall inequality, Lemma \ref{path} and \eqref{m1}, we get
\begin{equation}
\begin{aligned}
&\left(\mathbb{E}|\Psi^{\varepsilon}_{t}|^{p}\right)^{\frac{1}{p}}\leq Ce^{-\left(2-L\varepsilon^{1-\frac{1}{\alpha}}\right)(t-k\delta)/\varepsilon}+\frac{C\delta ^{\theta}}{2}e^{C(t-k\delta)/\varepsilon}.\\
\end{aligned}
\end{equation}

Integrating from $k\delta$ to $(k+1)\delta$ with respect to the variable $t$ in the above leads to

\begin{equation}
\label{psi}
\int^{(k+1)\delta}_{k\delta}\left(\mathbb{E}|\Psi^{\varepsilon}_{t}|^{p}\right)^{\frac{1}{p}}dt
\leq C_{T}\left( \varepsilon+\varepsilon\delta^{\theta}e^{{C\delta}/{\varepsilon}}\right),
\end{equation}
then we have
\begin{equation}
\label{yy}
\int_{0}^{T}\left(\mathbb{E}|Y^{\varepsilon}_{t}-\tilde{Y}^{\varepsilon}_t|^{p}\right)^{\frac{1}{p}}dt
\leq C_{T}\left(\frac{\varepsilon}{\delta}+\varepsilon {\delta}^{-(1-\theta)}e^{C\delta/{\varepsilon}}\right).
\end{equation}
Taking \eqref{yy} and Lemma \ref{path} into \eqref{xx}, it yields that
\begin{equation}
\int_{0}^{T}\left(\mathbb{E}|X^{\varepsilon}_{t}-\tilde{X}^{\varepsilon}_t|^{p}\right)^{\frac{1}{p}}dt
\leq C_{T}\left(\delta^{\theta}+\frac{\varepsilon}{\delta}+\varepsilon \delta^{-(1-\theta)}e^{C\delta/\varepsilon}\right).
\end{equation}

\end{proof}
\end{lemma}

\begin{lemma}
\label{lipc}
Under Hypotheses $\bf{H.1}$-$\bf{H.3}$, there exists a positive constant $\varepsilon_0$, such that for any $\varepsilon \in (0, \varepsilon_0)$, the functions $\bar{F}$ and $\bar{G}$ satisfy Lipschitz condition, where
\begin{equation}
\label{0c}
\begin{aligned}
&\bar{F}(x,\varepsilon)=\int F(x,y,\varepsilon)\mu^{\varepsilon}_{x}(dy), ~~\bar{F}(x)=\lim_{\varepsilon\rightarrow 0}\bar{F}(x,\varepsilon), \\
&\bar{G}(x,\varepsilon)=\int G(x,y,\varepsilon)\mu^{\varepsilon}_{x}(dy), ~~\bar{G}(x)=\lim_{\varepsilon\rightarrow 0}\bar{G}(x,\varepsilon). \\
\end{aligned}
\end{equation}
\end{lemma}
\begin{proof}
As $\mu^{\varepsilon}_{x}$ is ergodic, for any $h\in \mathbb{R}^{n}$, $x_1, x_2 \in \mathbb{R}^{n}$ and $t>0$, we have
\begin{equation}
\label{0f}
\begin{aligned}
&\frac{1}{t}\left|\langle F(x_1, Y^{\varepsilon}_{t,x_1}, \varepsilon)-F(x_2, Y^{\varepsilon}_{t,x_2}, \varepsilon),h\rangle\right|\\
&\leq \frac{L}{t}\int^{t}_{0}\left[|x_1-x_2|+|Y^{\varepsilon}_{s,x_1}- Y^{\varepsilon}_{s,x_2}|\right]ds \cdot |h|\\
&\leq L|h|\left(|x_1-x_2|+\sup_{x,y,\varepsilon}|\nabla_{x}Y^{\varepsilon}_{t,x}||x_1-x_2|\right).
\end{aligned}
\end{equation}
Hence, thank to Hypothesis $\bf{H.3}$ and \cite[Theorem 1.1]{zh}, it is immediate to check that for any $t \in [0,T]$, we have
\begin{equation}
\label{0g}
\sup_{x,y,\varepsilon}|\nabla_{x}Y^{\varepsilon}_{t,x}|\leq C_{T}, ~~\mathbb{P}-a.s.
\end{equation}
Combined with $\eqref{0f}$ and $\eqref{0g}$, we have
\begin{equation}
\frac{1}{t}\left|\langle F(x_1, Y^{\varepsilon}_{t,x_1}, \varepsilon)-F(x_2, Y^{\varepsilon}_{t,x_2}, \varepsilon),h\rangle\right|\leq C|h||x_1-x_2|.
\end{equation}
Therefore we can conclude that $\bar{F}(x,\varepsilon)$ is Lipschitz. According to \eqref{0c} and \eqref{0f} , we have
\begin{equation}
\begin{aligned}
|\bar{F}(x_1)-\bar{F}(x_1)|& \leq |\bar{F}(x_1)-\bar{F}(x_1, \varepsilon)|+|\bar{F}(x_1,\varepsilon)-\bar{F}(x_2, \varepsilon)|+|\bar{F}(x_2)-\bar{F}(x_2, \varepsilon)| \\
& \leq C|x_1-x_2|.
\end{aligned}
\end{equation}
Similarly, the function $\bar{G}$ also satisfies Lipschitz condition.
\end{proof}

\begin{theorem}
Under Hypotheses $\bf{H.1}$-$\bf{H.3}$,  the slow component $X^{\varepsilon}_t$ converges to $\bar{X}_t$ in $L^{p}(1<p<\alpha)$, i.e.,
\begin{equation}
\lim_{\varepsilon\rightarrow 0}\mathbb{E}|X^{\varepsilon}_t-\bar{X}_t|^{p}=0,
\end{equation}
and
\begin{equation}
\label{average}
d\bar{X}_t=\frac{1}{2}\bar{F}(\bar{X}_t)dt+\frac{1}{2}\bar{G}(\bar{X}_t)dt+\frac{(\sigma_1+\sigma_2)}{2}dL^{\alpha}_t,
\end{equation}
\end{theorem}

\begin{proof}
Step 1. By the equations \eqref{slo} and \eqref{average}, we have
\begin{equation}
\begin{aligned}
\left(\mathbb{E}|X^{\varepsilon}_t-\bar{X}_t|^{p}\right)^{\frac{1}{p}}&\leq \left(\mathbb{E}|X^{\varepsilon}_t-\tilde{{X}}^{\varepsilon}_t|^{p}\right)^{\frac{1}{p}}+\left(\mathbb{E}|\tilde{{X}}^{\varepsilon}_t-\bar{{X}}_t|^{p}\right)^{\frac{1}{p}}\\
&\leq C_{T} \left(\delta^{\theta}+\frac{\varepsilon}{\delta}+\varepsilon \delta^{-(1-\theta)}e^{C\delta/\varepsilon}\right)+\left(\mathbb{E}|\tilde{{X}}^{\varepsilon}_t-\bar{{X}}_t|^{p}\right)^{\frac{1}{p}}.
\end{aligned}
\end{equation}
Meanwhile,
\begin{equation}
\label{req}
\begin{aligned}
\left(\mathbb{E}|\tilde{{X}}^{\varepsilon}_t-\bar{{X}}_t|^{p}\right)^{\frac{1}{p}}
&=\left(\mathbb{E}\left|\frac{1}{2}\int^{t}_{0}\left[F\left(X^{\varepsilon}_{[{s}/{\delta}]\delta},\tilde{{Y}}^{\varepsilon}_s,\varepsilon\right)
-\bar{F}(\bar{X}_{s})\right]ds+\frac{1}{2}\int^{t}_{0}\left[G\left(X^{\varepsilon}_{[{s}/{\delta}]\delta},\tilde{{Y}}^{\varepsilon}_s,\varepsilon\right)
-\bar{G}(\bar{X}_{s})\right]ds\right|^{p}\right)^{\frac{1}{p}}\\
&\leq \int^{t}_{0}\left(\mathbb{E}\left|F\left(X^{\varepsilon}_{[{s}/{\delta}]\delta},\tilde{{Y}}^{\varepsilon}_s,\varepsilon\right)
-\bar{F}(\bar{X}_{s})\right|^{p}\right)^{\frac{1}{p}}ds+\int^{t}_{0}\left(\mathbb{E}\left|G\left(X^{\varepsilon}_{[{s}/{\delta}]\delta},\tilde{{Y}}^{\varepsilon}_s,\varepsilon\right)
-\bar{G}(\bar{X}_{s})\right|^{p}\right)^{\frac{1}{p}}ds\\
&\leq \int^{t}_{0}\left(\mathbb{E}\left|F\left(X^{\varepsilon}_{[{s}/{\delta}]\delta},\tilde{{Y}}^{\varepsilon}_s,\varepsilon\right)
-\bar{F}({X}^{\varepsilon}_{[{s}/{\delta}]\delta})\right|^{p}\right)^{\frac{1}{p}}ds
+\int^{t}_{0}\left(\mathbb{E}\left|\bar{F}\left({X}^{\varepsilon}_{[{s}/{\delta}]\delta}\right)
-\bar{F}(X^{\varepsilon}_{s})\right|^{p}\right)^{\frac{1}{p}}ds\\
&+\int^{t}_{0}\left(\mathbb{E}\left|\bar{F}({X}^{\varepsilon}_{s})
-\bar{F}(\tilde{{X}}^{\varepsilon}_{s})\right|^{p}\right)^{\frac{1}{p}}ds
+\int^{t}_{0}\left(\mathbb{E}\left|\bar{F}(\tilde{X}^{\varepsilon}_{s})
-\bar{F}(\bar{{X}}_{s})\right|^{p}\right)^{\frac{1}{p}}ds\\
&+\int^{t}_{0}\left(\mathbb{E}\left|G\left(X^{\varepsilon}_{[{s}/{\delta}]\delta},\tilde{{Y}}^{\varepsilon}_s,\varepsilon\right)
-\bar{G}({X}^{\varepsilon}_{[{s}/{\delta}]\delta})\right|^{p}\right)^{\frac{1}{p}}ds
+\int^{t}_{0}\left(\mathbb{E}\left|\bar{G}\left({X}^{\varepsilon}_{[{s}/{\delta}]\delta}\right)
-\bar{G}\left(X^{\varepsilon}_{s}\right)\right|^{p}\right)^{\frac{1}{p}}ds \\
&+\int^{t}_{0}\left(\mathbb{E}\left|\bar{G}({X}^{\varepsilon}_{s})
-\bar{G}(\tilde{{X}}^{\varepsilon}_{s})\right|^{p}ds\right)^{\frac{1}{p}}
+\int^{t}_{0}\left(\mathbb{E}\left|\bar{G}(\tilde{X}^{\varepsilon}_{s})
-\bar{G}(\bar{{X}}_{s})\right|^{p}\right)^{\frac{1}{p}}ds\\
&=:I_1+I_2+I_3+I_4+J_1+J_2+J_3+J_4.
\end{aligned}
\end{equation}
For $I_1$, we have
\begin{equation}
\label{i1}
\begin{aligned}
&\int^{t}_{0}\left(\mathbb{E}\left|F\left(X^{\varepsilon}_{[{s}/{\delta}]\delta},\tilde{{Y}}^{\varepsilon}_s,\varepsilon\right)
-\bar{F}({X}^{\varepsilon}_{[{s}/{\delta}]\delta})\right|^{p}\right)^{\frac{1}{p}}ds \\
& \leq \sum^{[{t}/{\delta}]}_{0}\left(\mathbb{E}\left|\int^{(k+1)\delta}_{k\delta}\left(F\left(X^{\varepsilon}_{k\delta},\tilde{{Y}}^{\varepsilon}_s,\varepsilon\right)
-\bar{F}\left({X}^{\varepsilon}_{k\delta}\right)\right)ds\right|^{p}\right)^{\frac{1}{p}}\\
&\leq \varepsilon \sum^{[{t}/{\delta}]}_{0}\left(\int^{\frac{\delta}{\varepsilon}}_{0}\int^{\frac{\delta}{\varepsilon}}_{s}\Upsilon_{k}(r,s)drds\right)^{\frac{1}{2}},
\end{aligned}
\end{equation}
where $t:=\left([t/\delta]+1\right)\delta$ and
\begin{equation}
\begin{aligned}
\Upsilon_{k}(r,s)&=\mathbb{E}\left\langle \left(F\left(X^{\varepsilon}_{k\delta},\tilde{{Y}}^{\varepsilon}_{r\varepsilon+k\delta},\varepsilon\right)
-\bar{F}\left({X}^{\varepsilon}_{k\delta}\right)\right),  \left(F\left(X^{\varepsilon}_{k\delta},\tilde{{Y}}^{\varepsilon}_{s\varepsilon+k\delta},\varepsilon\right)
-\bar{F}\left({X}^{\varepsilon}_{k\delta}\right)\right)\right\rangle.
\end{aligned}
\end{equation}
For any $s\in(0,\delta)$, from the equation \eqref{xy}, we have
\begin{equation}
\begin{aligned}
\tilde{Y}^{\varepsilon}_{s+k\delta}&=e^{\frac{-2s}{\varepsilon}}\tilde{Y}^{\varepsilon}_{k\delta}+\frac{1}{\varepsilon}\int_{0}^{s}e^{\frac{-2(s-u)}{\varepsilon}}
\left[\frac{1}{2}\varepsilon^{1-\frac{1}{\alpha}}\left(F\left({X}^{\varepsilon}_{k\delta},\tilde{Y}^{\varepsilon}_{k\delta+u},\varepsilon\right)
-G\left({X}^{\varepsilon}_{k\delta},\tilde{Y}^{\varepsilon}_{k\delta+u},\varepsilon\right)\right)\right]ds \\
&+\frac{\sigma_1-\sigma_2}{2\varepsilon^{\frac{1}{\alpha}}}\int_{0}^{s}e^{\frac{-2(s-u)}{\varepsilon}}d\widetilde{L}^{\alpha}_{u},
\end{aligned}
\end{equation}
where $\widetilde{L}^{\alpha}_{.}=\widetilde{L}^{\alpha}_{.+k\delta}-\widetilde{L}^{\alpha}_{k\delta}$ with filtration $\mathcal{F}_{.+k\delta}$.

Define the process $Y^{X^{\varepsilon}_{k\delta}, \tilde{Y}^{\varepsilon}_{k\delta}}_{s}$ by
\begin{equation}
\begin{aligned}
Y^{X^{\varepsilon}_{k\delta}, \tilde{Y}^{\varepsilon}_{k\delta}}_{\frac{s}{\varepsilon}}&=e^{\frac{-2s}{\varepsilon}}\tilde{Y}^{\varepsilon}_{k\delta}+\int^{\frac{s}{\varepsilon}}_{0}
e^{-2(s/\varepsilon-u)}\left[\frac{1}{2}\varepsilon^{1-\frac{1}{\alpha}}\left(F\left({X}^{\varepsilon}_{k\delta},{Y}^{X^{\varepsilon}_{k\delta}, \tilde{Y}^{\varepsilon}_{k\delta}}_{u},\varepsilon\right)
-G\left({X}^{\varepsilon}_{k\delta},{Y}^{X^{\varepsilon}_{k\delta}, \tilde{Y}^{\varepsilon}_{k\delta}}_{u},\varepsilon\right)\right)\right]du\\
&+\frac{\sigma_1-\sigma_2}{2}\int^{\frac{s}{\varepsilon}}_{0}e^{-2(s/\varepsilon-u)}d\widehat{{L}}^{\alpha}_{u},
\end{aligned}
\end{equation}
where $\widehat{{L}}^{\alpha}_{t}$ is symmetric $\alpha$-stable L\'evy processes. Moreover, $\widehat{{L}}^{\alpha}_{t}$ is independent of $L^{\alpha}_{t}$ and $\widetilde{L}^{\alpha}_{t}$.
\end{proof}
After a series of simple calculations, we have
\begin{equation}
\begin{aligned}
Y^{X^{\varepsilon}_{k\delta}, \tilde{Y}^{\varepsilon}_{k\delta}}_{\frac{s}{\varepsilon}}&=e^{\frac{-2s}{\varepsilon}}\tilde{Y}^{\varepsilon}_{k\delta}+\int^{s}_{0}
e^{-2(s-u)/\varepsilon}\left[\frac{1}{2}\varepsilon^{1-\frac{1}{\alpha}}\left(F\left({X}^{\varepsilon}_{k\delta},{Y}^{X^{\varepsilon}_{k\delta}, \tilde{Y}^{\varepsilon}_{k\delta}}_{\frac{u}{\varepsilon}},\varepsilon\right)
-G\left({X}^{\varepsilon}_{k\delta},{Y}^{X^{\varepsilon}_{k\delta}, \tilde{Y}^{\varepsilon}_{k\delta}}_{\frac{u}{\varepsilon}},\varepsilon\right)\right)\right]\\
&+\frac{\sigma_1-\sigma_2}{2\varepsilon^{\alpha}}\int^{s}_{0}e^{-2(s-u)/\varepsilon}d\check{{L}}^{\alpha}_{u},
\end{aligned}
\end{equation}
where $\check{{L}}^{\alpha}_{\cdot}=\varepsilon^{1/\alpha}\hat{L}^{\alpha}_{\cdot/\varepsilon}$.

By the self-similar property of $\alpha$-stable process, we know $\widetilde{Y}^{\varepsilon}_{s+k\delta}$ and $Y^{X^{\varepsilon}_{k\delta}, \widetilde{Y}^{\varepsilon}_{k\delta}}_{\frac{s}{\varepsilon}}$ have the same distribution, i.e.,
\begin{equation}
\label{dis}
\mathbb{P}(\widetilde{Y}^{\varepsilon}_{s+k\delta})=\mathbb{P}(Y^{X^{\varepsilon}_{k\delta}, \widetilde{Y}^{\varepsilon}_{k\delta}}_{\frac{s}{\varepsilon}}).
\end{equation}
Let
\begin{equation}
\mathcal{F}_{s}:=\sigma\{Y^{X^{\varepsilon}_{k\delta}, \widetilde{Y}^{\varepsilon}_{k\delta}}_{u},  u\leq s\},
\end{equation}
then for $r>s$, by the property of conditional expectation, Hypothesis $\bf{H.1}$ and Lemma \ref{lipc},  we have
\begin{equation}
\begin{aligned}
&\Upsilon_{k}(r,s)=\mathbb{E}\left\langle \left(F\left(X^{\varepsilon}_{k\delta},\tilde{{Y}}^{\varepsilon}_{r\varepsilon+k\delta},\varepsilon\right)
-\bar{F}\left({X}^{\varepsilon}_{k\delta}\right)\right),  \left(F\left(X^{\varepsilon}_{k\delta},\tilde{{Y}}^{\varepsilon}_{s\varepsilon+k\delta},\varepsilon\right)
-\bar{F}\left({X}^{\varepsilon}_{k\delta}\right)\right)\right\rangle \\
&=\mathbb{E}\left\langle \left(F\left(X^{\varepsilon}_{k\delta},{{Y}}^{X^{\varepsilon}_{k\delta}, \widetilde{Y}^{\varepsilon}_{k\delta}}_{r},\varepsilon\right)
-\bar{F}\left({X}^{\varepsilon}_{k\delta}\right)\right),  \left(F\left(X^{\varepsilon}_{k\delta},{{Y}}^{X^{\varepsilon}_{k\delta}, \widetilde{Y}^{\varepsilon}_{k\delta}}_{s},\varepsilon\right)
-\bar{F}\left({X}^{\varepsilon}_{k\delta}\right)\right)\right\rangle \\
&=\mathbb{E}\left\langle \left(F\left(X^{\varepsilon}_{k\delta},{{Y}}^{X^{\varepsilon}_{k\delta}, \widetilde{Y}^{\varepsilon}_{k\delta}}_{s},\varepsilon\right)
-\bar{F}\left({X}^{\varepsilon}_{k\delta}\right)\right)\left.\left(\mathbb{E}\left(F\left(z_1,{{Y}}^{X^{\varepsilon}_{k\delta}, {Y}^{\varepsilon}_{k\delta}}_{r-s}+z_2,\varepsilon\right)
-\bar{F}\left({X}^{\varepsilon}_{k\delta}\right)\right)\right)\right|_{z_2={{Y}}^{X^{\varepsilon}_{k\delta}, \widetilde{Y}^{\varepsilon}_{k\delta}}_{s}}^{z_1=X^{\varepsilon}_{k\delta}}\right\rangle \\
&= \mathbb{E}\left\langle\left( F\left(X^{\varepsilon}_{k\delta},{{Y}}^{X^{\varepsilon}_{k\delta}, \widetilde{Y}^{\varepsilon}_{k\delta}}_{s},\varepsilon\right)
-\bar{F}\left({X}^{\varepsilon}_{k\delta}\right)\right)\cdot \left.\left(\mathbb{E}\left(F\left(z_1,{{Y}}^{X^{\varepsilon}_{k\delta}, {Y}^{\varepsilon}_{k\delta}}_{r-s}+z_2,\varepsilon\right)
-\int\bar{F}\left({X}^{\varepsilon}_{k\delta}\right)\right)\right)\right|_{z_2={{Y}}^{X^{\varepsilon}_{k\delta}, \widetilde{Y}^{\varepsilon}_{k\delta}}_{s}}^{z_1=X^{\varepsilon}_{k\delta}}\right\rangle \\
&\leq \left\{C_{q}\mathbb{E}\left|F\left({X}^{\varepsilon}_{k\delta}, Y_s^{X^{\varepsilon}_{k\delta}, \widetilde{Y}^{\varepsilon}_{k\delta}},\varepsilon\right)\right|^{q}+C_q \mathbb{E}\left|\int F\left({X}^{\varepsilon}_{k\delta},y\right)\mu^\epsilon_{x}(dy) \right|^{q}\right\}^{\frac{1}{q}} \\
& ~~~~\left\{\mathbb{E}\left.\left(\left|\left(\mathbb{E}\left(F\left(z_1,{{Y}}^{X^{\varepsilon}_{k\delta}, \widetilde{Y}^{\varepsilon}_{k\delta}}_{r-s}+z_2,\varepsilon\right)-\bar{F}\left(z_1\right)\right)\right)\right|_{z_2={{Y}}^{X^{\varepsilon}_{k\delta}, \widetilde{Y}^{\varepsilon}_{k\delta}}_{s}}^{z_1=X^{\varepsilon}_{k\delta}}\right|^{p}\right)\right\}^{\frac{1}{p}}\\
&\leq \left\{\left(L+\bar{L}\right)C_{q}\mathbb{E}\left|X^{\varepsilon}_{k\delta}\right|^{q}+L C_{q}\mathbb{E}\left|Y^{X^{\varepsilon}_{k\delta},\widetilde{Y}^{\varepsilon}_{k\delta}}_{s}\right|^{q}\right\}^{\frac{1}{q}}\cdot \\
&~~~~\left\{\mathbb{E}\left.\left(\left|\left(\mathbb{E}\left(F\left(z_1,{{Y}}^{X^{\varepsilon}_{k\delta}, \widetilde{Y}^{\varepsilon}_{k\delta}}_{r-s}+z_2,\varepsilon\right)-\bar{F}\left(z_1\right)\right)\right)\right|_{z_2={{Y}}^{X^{\varepsilon}_{k\delta}, \widetilde{Y}^{\varepsilon}_{k\delta}}_{s}}^{z_1=X^{\varepsilon}_{k\delta}}\right|^{p}\right)\right\}^{\frac{1}{p}}\\
 &\leq C \mathbb{E}\left.\left(\left|\left(\mathbb{E}\left(F\left(z_1,{{Y}}^{X^{\varepsilon}_{k\delta}, \widetilde{Y}^{\varepsilon}_{k\delta}}_{r-s}+z_2,\varepsilon\right)-\bar{F}\left(z_1\right)\right)\right)\right|_{z_2={{Y}}^{X^{\varepsilon}_{k\delta}, \widetilde{Y}^{\varepsilon}_{k\delta}}_{s}}^{z_1=X^{\varepsilon}_{k\delta}}\right|^{p}\right)^{\frac{1}{p}}.
\end{aligned}
\end{equation}
By Lemma \ref{avv}, Lemma 2 and \eqref{bounded},  we have
\begin{equation}
\label{rs}
\begin{aligned}
\Upsilon_{k}(r,s)
&\leq C e^{-\frac{2(r-s)}{\varepsilon}}\mathbb{E}\left(|X^{\varepsilon}_{k\delta}|
+|\tilde{Y}^{X^{\varepsilon}_{k\delta},\widetilde{Y}^{\varepsilon}_{k\delta}}_{s}|\right)\\
&\leq C e^{-\frac{2(r-s)}{\varepsilon}}.
\end{aligned}
\end{equation}
Taking \eqref{rs} into \eqref{i1}, we have
\begin{equation}
\label{I_1}
I_1\leq C \varepsilon.
\end{equation}
For $I_2$, by the Lipschitz property of $\bar{F}$, we have
\begin{equation}
I_2\leq  C \delta^{\theta}.
\end{equation}
For $I_3$, by the Lipschitz property of $\bar{F}$ and Lemma \ref{lem5}, we have
\begin{equation}
I_3\leq  C\left(\delta^{\theta}+\frac{\varepsilon}{\delta}+\varepsilon \delta^{-(1-\theta)}e^{C\delta/\varepsilon}\right).
\end{equation}
For $I_4$, by the Lipschitz property of $\bar{F}$, we have
\begin{equation}
I_4\leq\int^{t}_{0}\left(\mathbb{E}\left|\tilde{X}^{\varepsilon}_{s}
-\bar{{X}}_{s}\right|^{p}\right)^{\frac{1}{p}}ds.
\end{equation}
Step 2.
For $J_1, J_2, J_3, J_4$, taking the same method as $I_1, I_2, I_3, I_4$ respectively, we have
\begin{equation}
\label{J}
\begin{aligned}
&J_1\leq C \varepsilon, ~~J_2 \leq  C  \delta^{\theta},\\
&J_3\leq  C \left( \frac{\varepsilon}{\delta}+\varepsilon \delta^{-(1-\theta)}e^{C\delta/\varepsilon}\right), \\
&J_4\leq  C \left(\int^{t}_{0}\left(\mathbb{E}\left|\tilde{X}^{\varepsilon}_{s}
-\bar{{X}}_{s}\right|^{p}\right)^{\frac{1}{p}}ds\right).
\end{aligned}
\end{equation}
Step 3. Taking \eqref{I_1}-\eqref{J} into \eqref{req}, and applying Gr\"onwall inequality in \eqref{req}, we obtain
\begin{equation}
\left(\mathbb{E}\left|X^{\varepsilon}_{t}-\bar{X}(t)\right|^{p}\right)^{\frac{1}{p}}\leq C\left( \delta^{\theta}+\frac{\varepsilon}{\delta}+\sqrt{\frac{\varepsilon}{\delta}}+\varepsilon\delta^{-(1-\theta)}e^{C\delta/\varepsilon}\right).
\end{equation}
Let $\delta=\varepsilon(-\ln\varepsilon)^{1/2}$ and taking $\varepsilon\rightarrow 0$, we yields
\begin{equation}
\lim_{\varepsilon\rightarrow 0}\left(\mathbb{E}|X^{\varepsilon}_t-\bar{X}_t|^{p}\right)^{\frac{1}{p}}=0.
\end{equation}
\begin{remark}
It is noteworthy that, equation \eqref{dis} is valid only for the finite dimensional stochastic differential equation \cite{ZB}.
\end{remark}

\section{The persistence of synchronization}

 In the following, we will illustrate the synchronization occurs, i.e.,  the  random attractor is a singleton.

\begin{lemma}
Under Hypotheses $\bf{H.1}$-$\bf{H.{3}}$, for fixed $Y^{\varepsilon}_{t}=y$, the random dynamical system $\varphi$ associated with \eqref{slo} have a compact random attractor $\mathcal{A}(\omega)$ consisting of a single point: $\mathcal{A}(\omega)=\{\eta_{0}(\omega)\}$.
\end{lemma}
\begin{proof}
The proof is similar to \cite[Theorem 1.12]{wei}.
\end{proof}

Note that if the random attractor consists of singleton sets, i.e $ A(\omega) =\{\eta(\omega)\}$ for some random variable $\eta$, then $\eta_t(\omega) := \eta(\theta_t\omega)$ is a stationary stochastic process.

%\begin{remark}
%For a stochastic differential equation with initial value $X_0=x$, the solution can be represented as
%\begin{equation}
%X_t(\omega)=\varphi(t,\omega,x)
%\end{equation}
%\end{remark}

\begin{lemma}
\label{sta1}
Under Hypotheses $\bf{H.1}$-$\bf{H.{3}}$, let $\hat{\eta}^{\varepsilon}_t$ be the stationary stochastic solution of the slow component $X^{\varepsilon}_t$, and $\bar{\eta}_t$ be the stationary stochastic solution of $\bar{X}_t$, then we have
\begin{equation}
\lim_{\varepsilon\rightarrow 0}\mathbb{E}|\hat{\eta}^{\varepsilon}_t-\bar{\eta}_t|^{p}=0.
\end{equation}
\end{lemma}
\begin{proof}
By the definition of stationary stochastic solution,  we know $\hat{\eta}^{\varepsilon}_t$ and $ \bar{\eta}_t$ satisfy the following equations, i.e.,
\begin{equation}
d\hat{\eta}^{\varepsilon}_t=\frac{1}{2}\left[F\left(\hat{\eta}^{\varepsilon}_t, Y^{\varepsilon}_t, \varepsilon)
+G(\hat{\eta}^{\varepsilon}_t,Y^{\varepsilon}_t, \varepsilon\right)\right]dt+\frac{\sigma_1+\sigma_2}{2}dL^{\alpha}_t,
\end{equation}
and
\begin{equation}
d\bar{\eta}_t=\frac{1}{2}\bar{F}(\bar{\eta}_t)dt+\frac{1}{2}\bar{G}(\bar{\eta}_t)dt+\frac{(\sigma_1+\sigma_2)}{2}dL^{\alpha}_t.
\end{equation}
By the similar method as Theorem 1, we have
\begin{equation}
\lim_{\varepsilon\rightarrow 0}\mathbb{E}|\hat{\eta}^{\varepsilon}_t-\bar{\eta}_t|^{p}=0.
\end{equation}
\end{proof}

\begin{lemma}
\label{sta2}
Under Hypotheses $\bf{H.1}$-$\bf{H.{3}}$, for fixed $X^{\varepsilon}_{t}=x$, there exists a positive constant $M$ independent of the scale parameter $\varepsilon$,  such that the stationary stochastic solution $\chi^{\varepsilon}_t$ of $Y^{\varepsilon}_t$ in \eqref{sloy} is uniformly bounded, i.e.,
\begin{equation}
\mathbb{E}|\chi^{\varepsilon}_t|^{p}\leq M.
\end{equation}
\begin{proof}
By the definition of random dynamical system, we have
\begin{equation}
\chi^{\varepsilon}_t=\chi^{\varepsilon}(\theta_t\omega)=\phi(t, \omega, \chi^{\varepsilon}\omega).
\end{equation}
By Lemma 2, we yields
\begin{equation}
\mathbb{E}|\chi^{\varepsilon}_t|^{p}\leq M.
\end{equation}
\end{proof}
\end{lemma}

The aim is to show that this synchronization effect is preserved under additive $\alpha$-stable noise provided equilibria are replaced by stationary random solutions, and obtain the convergence rate of synchronization of the coupled systems.

\begin{theorem}
\label{sty}
Under Hypotheses $\bf{H.1}$-$\bf{H.{3}}$, for any given time intervals $[T_1,T_2]$ of $\mathbb{R}$, the stochastic dynamical system \eqref{syn} has a unique stationary stochastic solution $(\bar{X}_t^{\nu},\bar{Y}_t^{\nu})$, which is  globally asymptotically stable with
\begin{equation}
\lim_{\nu \rightarrow \infty}\mathbb{E}\big[|\bar{X}_t^{\nu}-\hat{X}_t|^{p}+|\bar{Y}_t^{\nu}-\hat{X}_t|^{p}\big]=0,
\end{equation}
 where $\hat{X}_t$  is the unique  globally asymptotically stable stationary stochastic solution of the
 following `averaged' SDE
 \begin{equation}
 d\bar{X}_t=\frac{1}{2}f(\bar{X}_t)dt+\frac{1}{2}g(\bar{X}_t)dt+\frac{(\sigma_1+\sigma_2)}{2}dL^{\alpha}_t.
 \end{equation}
\end{theorem}
\begin{proof}
By the relationship between $(X_t, Y_t)$ and $(X^{\varepsilon}_t,Y^{\varepsilon}_t)$, we have
\begin{equation}
\left\{
\begin{aligned}
& X^{\varepsilon}_{t}=\frac{X_t+ Y_t}{2}, \\
& Y^{\varepsilon}_{t}=\frac{X_t-Y_t}{2\varepsilon^{\frac{1}{\alpha}}}.
\end{aligned}
\right.
\end{equation}
Therefore we only need to show the stationary stochastic solution
$\eta^{\varepsilon}_t$ of $X^{\varepsilon}_t$ converges to $\hat{X}_t$ with $\varepsilon\rightarrow 0$, and the stationary solution $\chi^{\varepsilon}_t$ of $Y^{\varepsilon}_t$ is uniformly bounded.
By Lemma \ref{sta1} and Lemma \ref{sta2}, we yields the above assertion.
\end{proof}
\section{Conclusion}
In this paper, we study the synchronization problem of nonlinear coupled dynamical systems driven by $\alpha$-stable noise. By introducing a new transformation, we construct the relationship between synchronized system and slow-fast system. Using the stochastic averaging method, we show that the synchronization effect is persisted provided equilibria are replaced by stationary random solutions. Furthermore, we will consider  the persistence of synchronization  phenomenon under observation data. This could potentially find applications in biology, physics and social science. It is also possible to discuss the connection between the  persistence of synchronization  and data assimilation. These topics are being studied and will be reported in future works.

\medskip
\textbf{Acknowledgements}.  We would like to thank Jicheng Liu (Huazhong University of Sciences and Technology, China) for helpful discussions.
 The research of Y. Zhang was supported by the NSFC grant 11901202.  The research of J. Duan was partly supported by the NSFC grant 11531006 and 11771449. The research of H. Fu was partly supported by the NSFC grant 11826209.

\end{document}